\documentclass{amsart}
\usepackage[top=3cm,bottom=2cm,left=3cm,right=3cm,marginparwidth=1.75cm]{geometry}

\usepackage{amsmath, amssymb, amsthm, enumerate}
\usepackage{mathtools}
\usepackage[alphabetic]{amsrefs}
\usepackage{graphicx}
\usepackage{relsize}

\usepackage{stmaryrd}

\usepackage{comment}

\usepackage[section]{algorithm}
\usepackage[noend]{algorithmic}

\makeatletter
\newenvironment{procedure}[1][htb]{%
	\renewcommand{\ALG@name}{Procedure}
	\begin{algorithm}[#1]%
	}{\end{algorithm}}
\makeatother

\usepackage{hyperref}

\makeatletter
\renewcommand\paragraph{\@startsection{paragraph}{4}{\z@}%
										{-3.25ex\@plus -1ex \@minus -.2ex}%
	                                    {1.5ex \@plus .2ex}%
	                                    {\normalfont\normalsize\bfseries}}
\renewcommand\subsubsection{\@startsection{subsubsection}{3}{\z@}%
										{-3.25ex\@plus -1ex \@minus -.2ex}%
	                                    {1.5ex \@plus .2ex}%
	                                    {\normalfont\normalsize\bfseries}}
\makeatother

\usepackage[utf8]{inputenc}

\usepackage{xcolor}

\newtheorem{theorem}{Theorem}[section]
\newtheorem{lemma}[theorem]{Lemma}
\newtheorem{proposition}[theorem]{Proposition}
\newtheorem*{proposition*}{Proposition}
\newtheorem*{theorem*}{Theorem}
\newtheorem{corollary}[theorem]{Corollary}
\newtheorem*{corollary*}{Corollary}

\newtheorem{problem}[theorem]{Problem}

\theoremstyle{definition}
\newtheorem{definition}[theorem]{Definition}
\newtheorem{example}[theorem]{Example}
\newtheorem{proposition/definition}[theorem]{Proposition/Definition}
\newtheorem{remark}[theorem]{Remark}

\newtheorem*{warning*}{Warning}

\numberwithin{theorem}{subsection}

\newcommand{\mbz}{\mathbb{Z}}
\newcommand{\mbr}{\mathbb{R}}
\newcommand{\Fp}{\mathbb{F}_p}
\newcommand{\Zp}{\mathbb{Z}_p}
\newcommand{\Qp}{\mathbb{Q}_p}
\newcommand{\wtilde}{\widetilde}

\newcommand{\abs}[1]{\left \vert #1 \right \vert}
\newcommand{\norm}[1]{\left \Vert #1 \right \Vert}
\newcommand{\gen}[1]{\left \langle #1 \right \rangle}

\newcommand{\bbm}{\begin{bmatrix}}
\newcommand{\ebm}{\end{bmatrix}}

\DeclareMathOperator{\GL}{GL}
\DeclareMathOperator{\M}{M}
\DeclareMathOperator{\trace}{trace}

\DeclareMathOperator{\lker}{lker}
\DeclareMathOperator{\rank}{rank}

\definecolor{answer}{rgb}{0,0.5,0.2}
\definecolor{avianswer}{rgb}{1,0,0.2}

\DeclarePairedDelimiter{\ceil}{\lceil}{\rceil}

\newcommand{\sizesorted}{size-sorted}
\newcommand{\thingname}{{\sizesorted } Hessenberg matrix}
\newcommand{\unshiftedthingname}{sorted Hessenberg matrix}

\newcommand{\smeig}{\mathrm{small}}
\newcommand{\bigeig}{\mathrm{big}}

\newcommand{\ind}[1]{\{#1\}}

\newcommand{\blockHess}[3]{
	\left[ \begin{array}{c|c}
		#1 & \begin{array}{cc}
			\phantom{*} & \phantom{*} \\
			 \multicolumn{2}{c}{\smash{\raisebox{.5\normalbaselineskip}{\scalebox{1.5}{$\ast$}}}} \end{array} \\ \hline
		\begin{array}{cc} 0 & #2 \\ 0 & 0 \end{array} & #3 \\ 
	\end{array} \right]
}

\newcommand{\propIterationSubroutine}{
	Let $M := [A; \epsilon, B]$ be a \thingname, let $m=n_B$ and let $\lambda_1, \ldots, \lambda_m$ be the small eigenvalues of $M$. If $ \eta := \max_{i,j} \abs{\lambda_i - \lambda_j} \leq \abs{\epsilon}$, then after $m$ $QR$-rounds we 
	obtain a size-sorted Hessenberg matrix $[A_{\mathrm{next}}; \epsilon_{\mathrm{next}}, B_{\mathrm{next}}]$ such that $\abs{\epsilon_{\mathrm{next}}} \leq \abs{\epsilon^2}$. Each round uses $2n^2 + o(n^2)$ operations of $\Qp$ arithmetic.
	After at most $(m \lceil \log_2 (-\log_p \eta) \rceil)$ rounds, the obtained $[A_{\mathrm{next}}; \epsilon_{\mathrm{next}}, B_{\mathrm{next}}]$ is such that $\vert \epsilon_{\mathrm{next}} \vert < \eta$.
}

\newcommand{\mainTheorem}{
	Let $M \in \M_n(\Zp)$ be a matrix whose entries are known with error $O(p^N)$. If the characteristic polynomial of $M$ modulo $p$ is square-free and factors completely then Algorithm~\ref{algo: main algorithm} computes a Schur form $T$ and a matrix $U \in \GL_n(\Zp)$ such that $MU = UT + O(p^N)$ in at most $\frac{2}{3}n^3 \log_2 N + o(n^3 \log_2 N)$ arithmetic operations in $\Zp$ at $N$-digits of precision. In particular, $T$ reveals all the eigenvalues of $M$ with error $O(p^N)$. 
	An additional $O(n^3)$ arithmetic operations in $\Qp$ is then enough to compute
	a $\Qp$-basis of eigenvectors with coefficients in $\Zp$. 
}

\begin{document}
\title{Super-linear convergence in the p-adic QR-algorithm}

\thanks{Avinash Kulkarni has been supported by the Simons Collaboration on Arithmetic Geometry, Number Theory, and Computation (Simons Foundation grant 550033) and by the Forschungsinitiative on Symbolic Tools at TU Kaiserslautern.}

%
\author{Avinash Kulkarni \and
Tristan Vaccon}


%
%
\address{Dartmouth College, Hanover, NH 03755, USA}
\email{avinash.a.kulkarni@dartmouth.edu}

\address{Univ. Limoges, CNRS, XLIM, UMR 7252, F-87000 Limoges, France}
\email{tristan.vaccon@unilim.fr}
\maketitle              
\begin{abstract}
The QR-algorithm is one of the most important
algorithms in linear algebra. Its several variants make feasible the computation
of the eigenvalues and eigenvectors of a numerical real or complex
matrix, even when the dimensions of the matrix are enormous.
The first adaptation of the QR-algorithm to local fields
was given by the first author in 2019. However, in this version the
rate of convergence is only linear and in some cases the decomposition
into invariant subspaces is incomplete. We present a refinement of this algorithm with a super-linear convergence
rate in many cases.
\end{abstract}

\section{Introduction} \label{sec: intro}

Eigenvalues and eigenvectors are ubiquitous throughout mathematics and industrial applications. Much attention has been directed towards developing algorithms to compute the eigenvectors of a finite precision real or complex matrix, whether for the purposes of making new computations feasible in research or for a more efficient product in industry. Given the successes of eigenvector methods in numerical linear algebra, one can hope that exciting novel applications can come from the comparatively unexplored area of finite precision $p$-adic linear algebra. In analogy, we refer to the subject as \emph{pnumerical linear algebra}. 

The topic of pnumerical linear algebra was first addressed in the latter half of the 20th century \cites{Dixon1982exact, Panayi1995leopolt}. Recently it has seen renewed interest \cites{Kedlaya2010differential, CRV2015linear, CRV2017characteristic}. A noteworthy application is a polynomial time algorithm for computing points on an algebraic curve over a finite field based on computing the characteristic polynomial of a $p$-adic matrix \cite{Kedlaya2001}. This is useful in practical cryptography to select curves with good properties for cryptosystems. Another application is in solving a $0$-dimensional system of polynomial equations over $\Qp$ \cites{Kulkarni2019, Berthomieu2012algebraic}. 


The first method for computing the eigenvectors of a $p$-adic (or real) matrix $M$ is the schoolbook algorithm, consisting of the following steps:
	\begin{enumerate}
		\item
			Compute a Hessenberg form for $M$. (Optimization for step \ref{classical: step: kernel}.)
		\item
			Compute the characteristic polynomial of $M$.
		\item
			Solve for the roots $\{\lambda_i\}$.
		\item \label{classical: step: kernel}
			Compute $\ker(M - \lambda_i I)$ for each $i$.
		\item[(4b)]
			(For block Schur form:) Compute $\ker (M - \lambda_i I)^{d_i}$ for some $d_i$. 
	\end{enumerate}
Over the reals, this is not the main algorithm used in practice since step (3) is numerically unstable. A similar difficulty is encountered $p$-adically, in that one needs to know the characteristic polynomial at the maximum possible $p$-adic precision in order to correctly compute the roots. In the worst case scenario, for an $n \times n$ input matrix given at $N$ digits of precision in each entry, one needs to compute the characteristic polynomial using arithmetic with $nN$ digits -- for examples see Section~\ref{sec: troublesome examples}. For practical considerations, one must be careful of the extra costs imposed by precision increases. Worse still, step~4b can fail to give the correct answer due to a lack of precision on the input (see Example~\ref{ex: insufficient precision causes GZE error}). Unlike $\mbr$, $p$-adic fields admit algebraic extensions of arbitrarily large degree. Consequently, the cost of doing arithmetic in an extension is potentially much more severe.

To represent the finite precision of the input, we will say that $M \in \M_n(\Qp)$ is known with (absolute) error $O(p^N)$ if we know the initial part $a_{-v}p^{-v} + \ldots + a_{N-1}p^{N-1} + O(p^N)$ of the $p$-adic expansion for each entry in the matrix $M$. We say that $A = B + O(p^N)$ if the $p$-adic expansion for every entry of $A-B$ up to the $p^N$ term is $0$. In Section~\ref{sec: background}, we discuss $p$-adic precision in more detail. We say that a matrix is in \emph{block Schur form} if it is block upper triangular and the characteristic polynomial of each diagonal block is irreducible. The main problem of this article is:
\begin{problem} \label{problem: main}
	Given an $n \times n$ matrix $M$ over $\Qp$, whose entries are known with error $O(p^N)$, compute a block Schur form $T$ for $M$ and a matrix $U$ such that $MU = UT + O(p^N)$.
\end{problem}
Of course, by passing to the splitting field of the characteristic polynomial, we can convert a block Schur form to a Schur form by triangularizing each of the blocks. In this article, we choose to use only $\Qp$-arithmetic. The benefit being that we procrastinate on doing expensive extension field arithmetic. When $M \in \M_n(\Zp)$, it is possible to put $M$ into a block Schur form via some $V \in \GL_n(\Zp)$.

\begin{theorem*}[\ref{thm: Schur forms via GLn(Zp) transforms}]
	Let $M \in \M_n(\Zp)$ and let $\chi_M = f_1 \cdots f_r$ be a factorization in $\Zp[t]$ where the factors are pairwise coprime in $\Qp[t]$.  Then there exists a $U \in \GL_n(\Zp)$ such that $UMU^{-1}$ is block-triangular with $r$ blocks; the $j$-th block accounts for the eigenvalues $\lambda$ such that $f_j(\lambda) = 0$.
\end{theorem*}

\noindent
Our method of proof is to combine standard arguments for the existence of canonical forms over a field with the notion of orthogonality, introduced in Schikhof~\cite{schikhof2006ultrametric} and discussed in Section~\ref{sec: background}. As an immediate corollary, we obtain a refinement of the decomposition of \cite[Theorem~4.3.11]{Kedlaya2010differential}.

\begin{corollary*}[\ref{cor: Newton decomposition}, Newton decomposition] 
	Let $M \in \M_n(\Zp)$ and let $\nu_1 \leq \ldots \leq \nu_r$ be the distinct valuations of the eigenvalues of $M$. Then there exists a $U \in \GL_n(\Zp)$ such that $UMU^{-1}$ is block-triangular with $r$ diagonal blocks; the $j$-th block accounts exactly for the eigenvalues of valuation $\nu_j$.
\end{corollary*}

We next show how to improve the iterative computation of the block Schur form introduced in \cite{Kulkarni2019}. Our main theorem is:

\begin{theorem*}[\ref{thm:main_theo}]
	\mainTheorem
\end{theorem*}

If the assumptions of the theorem above are not met, our algorithm will attempt to use the accelerated convergence strategy anyway. The timings in Section~\ref{sec: implementations} demonstrate a significant improvement over both the classical method and the basic $QR$-iteration in \cite{Kulkarni2019} in computing a weak block Schur form (see Definition~\ref{def: weak block Schur form}), \emph{even when the hypothesis on the characteristic polynomial is not satisfied}. Furthermore, the slowdown of convergence can be detected dynamically in Algorithm~\ref{algo: fast QR}. Should this occur we fallback to running the $QR$-iteration with linear convergence, as in \cite{Kulkarni2019}.  

We describe the layout of the article. For the remainder of Section~\ref{sec: intro}, we establish notation and then precisely state our results regarding Algorithm~\ref{algo: fast QR}. In Section~\ref{sec: background} we discuss the background needed in the article. In Section~\ref{sec: gze}, we prove Theorem~\ref{thm: Schur forms via GLn(Zp) transforms} and discuss the computation of sorted and \sizesorted\ forms. In Section~\ref{sec: technical improvements}, we discuss the improved $QR$-iteration; here we give Algorithm~\ref{algo: fast QR}. In Section~\ref{sec: main algo}, we combine our results to produce Algorithm~\ref{algo: main algorithm} and we also prove Theorem~\ref{thm:main_theo}. Finally, in Section~\ref{sec: implementations} we discuss the implementation of our algorithm and give some timings.

\subsection{Notation} \label{subsec:notations}

We denote by `$\ast$' a wildcard ($\Zp$-)integral entry or block of integral entries in a matrix. Generally, we will use the wildcard entries in upper-right blocks as they are not especially noteworthy in our analysis aside from the fact that they are integral. 
For a matrix $M$, we denote its left kernel by $\lker(M)$ and its characteristic polynomial by $\chi_M$. If $\chi_M(t) \in \Zp[t]$ we denote by $\chi_{M,p}$ the reduction of $\chi_M$ to the residue field. The standard basis vectors are denoted by $e_1, e_2, \ldots, e_n$. For a ring $R$, the ring of $n\times n$-matrices with entries in $R$ is denoted $\M_n(R)$. The $(i,j)$-th entry of a matrix is denoted $A\ind{i,j}$ and $A\ind{\bullet, j}$ denotes the $j$-th column. Our choice of notation deviates from the standard to improve the readability of expressions like $\abs{R_B^{(j)}\ind{1,1}}$.

The $p$-adic absolute value is denoted by $\abs{\cdot}$ and normalized so that $\abs{p} = p^{-1}$, for a vector $v$ we denote $\norm{v} := \max_{i} \abs{v_i}$, and for a matrix $A$ we denote $\norm{A} := \max_{i,j} \abs{A\ind{i,j}}$. For a polynomial $f := f_nx^n + \ldots + f_0$, we denote $\norm{f} := \max_{i} \abs{f_i}$. 

For a matrix $A$, we denote its smallest singular value by $\sigma_*(A)$. \textit{i.e.} its invariant factor with smallest norm. 
An eigenvalue $\lambda$ of $A \in \M_n(\Zp)$ is \emph{small} if $\abs{\lambda} < 1$, and \emph{big} otherwise.

\subsection{The iteration subroutine}

Our iteration subroutine is the heart of the main algorithm. In this last section of the introduction, we introduce some definitions to describe the input to the iteration subroutine, and state the results on its output.

\begin{definition}
	A matrix $\bbm A & * \\ E & B \ebm \in \M_n(\Zp)$ is called \emph{sorted} if for some $\lambda \in \Zp$ we have
	\[
	E \equiv 0 \pmod p, \quad \chi_B(t) \equiv (t-\lambda)^{n_B} \pmod p, \quad \text{and } \chi_A(\lambda) \not \equiv 0 \pmod p
	\]
	where $n_B$ is the number of columns of the square matrix $B$. In the special case that $\chi_B(t) \equiv t^{n_B} \pmod p$, we say that the matrix is \emph{\sizesorted}.
\end{definition}

If $M$ is a sorted matrix whose $B$-block has size $1$, then the shift $M - (M\ind{n,n}) I$ is a {\sizesorted } matrix. A {\unshiftedthingname } is a matrix which is both sorted and in Hessenberg form. Similarly, a {\thingname } is a {\sizesorted } matrix in Hessenberg form. As these matrices feature prominently in our discussion of the $QR$-algorithm, we give them a special notation.

\begin{definition}
	We denote by $[A;\epsilon, B]$ a {\unshiftedthingname } of the form
	\[
	[A;\epsilon, B] := \blockHess{A}{\epsilon}{B}, \qquad \text{with } A \in \M_{n_A}(\Zp), \ B \in \M_{n_B}(\Zp), \ \epsilon \in \Zp.
	\]
	The \emph{block sizes} of $[A; \epsilon, B]$ is the tuple  $(n_A, n_B)$.
	If only one of the block sizes is relevant, we use the wildcard character \emph{`$\ast$'} to hold the place of the other entry. 
\end{definition}

\begin{definition}
    Let $M \in \M_n(\Qp)$.
    A \emph{$QR$-round} (with shift $\mu$) is the computation consisting of the following steps applied to $M$:
    \begin{enumerate}[\ \ \ \ 1.]
    	\item
    		Compute a $QR$-factorization $M - \mu I = QR$
    	\item
    		Set $M_{\text{next}} := RQ + \mu I$
    \end{enumerate}
	If a value for the shift $\mu$ is not mentioned explicitly, we mean $\mu=0$ by default. It will always be clear from context to which matrix we apply the $QR$-round steps when we use the term.
\end{definition}

To clarify our terminology, the term \emph{$QR$-iteration} broadly refers to a process consisting of multiple $QR$-rounds applied to an input matrix, particularly when we do not wish to specify the shifts or the number of rounds for the sake of exposition. Alternatively, Algorithm~\ref{algo: fast QR}, which is titled \texttt{QR\_Iteration}, is a $QR$-iteration where the number of $QR$-rounds is determined in advance based on the input and the shifts are chosen deterministically during the iteration.

We now state our technical result regarding the convergence of the $QR$-iteration applied to a \thingname.

\begin{proposition*}[\ref{prop: QR complexity}]
	\propIterationSubroutine
\end{proposition*}

Note that if $\eta \leq p^{-N}$ (which vacuously occurs when $m=1$), we need at most $(m \lceil \log_2 N \rceil)$ $QR$-rounds (with shifting) to deflate $\epsilon$ to $0 + O(p^N)$.

\begin{remark} \label{rem: most meaning}
	If $A + O(p^N)$ is an $n \times n$-matrix whose entries are chosen with the uniform probability distribution on $[0, \ldots, p^N-1]$, then the limit as $n \rightarrow \infty$ of the probability that $\chi_{A}$ is square-free is at least $\frac{1-p^{-5}}{1+p^{-3}}$ \cite{Fulman2002random}.
\end{remark}



\section{Background} \label{sec: background}

\subsection{Precision and $QR$-factorizations}

We state some basic definitions for our discourse. We follow \cite{Kulkarni2019} for terminology, and direct the reader to \cites{CRV2015linear, precision_book, Kedlaya2010differential} for more details. We can identify a subgroup of $\GL_n(\Qp)$ where every matrix is well-conditioned, serving the analogous role to $\operatorname{O}_n(\mbr)$ in the real setting.

\begin{lemma} \label{lem: characterization of GLnZp}
	Let $A \in \GL_n(\Qp) \cap \M_n(\Zp)$. Then the following are equivalent:
	\begin{enumerate}[(a)]
		\item 
		$A \in \GL_n(\mbz_p)$
		
		\item
		$\norm{A} = \norm{A^{-1}} = 1$
		
		\item
		The roots of $\chi_A$ lie in $\overline{\Zp}^\times$.
			
	\end{enumerate}
\end{lemma}

\begin{proof}
	For (a) if and only if (b), it is direct consequence of the fact that $\norm{A}=\vert \sigma_1(A) \vert,$
	the first invariant factor.
	For (c), see \cite[Theorem 4.3.8]{Kedlaya2010differential}.
\end{proof}

\begin{proposition/definition}[$p$-adic $QR$-factorization] \label{cor: padic QR}
	Let $A \in \mbz_p^{n \times m}$ be a matrix. Then there exists a $Q \in \GL_n(\mbz_p)$ and an upper triangular matrix $R \in \mbz_p^{n \times m}$ such that $A = QR$. 
\end{proposition/definition}

\begin{proof}
	See \cite[Chapter 4]{Kedlaya2010differential}, or note this follows from the Iwasawa decomposition of $\GL_n(\Qp)$.
\end{proof}

For a matrix $M \in \M_n(\Zp)$, the $QR$-factorization is generally not unique. For example, if $M := QR$ is a $QR$-decomposition, and $U \in \GL_n(\Zp)$ is an upper triangular matrix, We have that $(QU)(U^{-1}R)$ is also an upper triangular matrix. The following type of $QR$-decomposition is well suited to understand the kernel and rank of a matrix.

\begin{definition}
	We say that $M = QR$ is a \emph{strict} $QR$-factorization if for each $i \geq 2$, the first non-zero entry of the $i$-th row of $R$ is strictly to the right of the first non-zero entry of the $(i-1)$-th row. That is, $R$ is a matrix in echelon form.
\end{definition}

Over $\Zp$, a strict $QR$-decomposition for $M$ reveals the rank of $M$ as the number of non-zero pivots. Unfortunately, with insufficient precision not all strict $QR$-forms of a matrix reveal the rank in this way -- we discuss this further in Example~\ref{ex: QR rank needs precision}.

Let $M \in \M_n(\Zp)$ be a matrix. The \emph{Smith normal form} for $M$ is a diagonal matrix $\Sigma$ such that the diagonal elements $\sigma_1, \ldots, \sigma_n$ satisfy $\abs{\sigma_1} \geq \ldots \geq \abs{\sigma_n}$ and $M = U\Sigma V$ for some $U, V \in \GL_n(\Zp)$. The Smith normal form is the pnumerical analogue of the singular value decomposition from standard numerical linear algebra.

\begin{definition}
	Let $M \in \M_{m \times n}(\Zp)$ be a matrix and let $M = U\Sigma V$, with $\Sigma$ the Smith normal form and $U\in \GL_m(\Zp), V \in \GL_n(\Zp)$. The \emph{$p$-adic singular value decomposition} of $M$ is the decomposition $M = U\Sigma V$. The \emph{singular values} of $M$ are sizes of the diagonal entries of $\Sigma$.
\end{definition}

Since we are never concerned with matrices over the reals, we will simply use the terms ``$QR$-decomposition/factorization'' or ``singular value decomposition'' without the $p$-adic prefix in the sequel.

\begin{remark}
	With pivots chosen with respect to the $p$-adic norm, several standard algorithms also work for matrices over $\Qp$. Specifically:
	\begin{enumerate}[(a)]
		\item
		The standard algorithm to compute the Hessenberg form computes a Hessenberg form \cite{CRV2017characteristic}.
		
		\item
		The standard algorithm to compute a $PLU$-decomposition computes a $PLU$-decomposition. Moreover, $P^{-1}L \in \GL_n(\Zp)$, so this is a $p$-adic $QR$-decomposition \cite[Chapter~4]{Kedlaya2010differential}.
		
		\item
		One can modify the algorithm in $(b)$ to allow column pivoting, and then factor $U = \Sigma V$ with $\Sigma$ diagonal and $V \in \GL_n(\Zp)$ to compute a $p$-adic singular value decomposition. See the proof of \cite[Theorem~4.3.4]{Kedlaya2010differential} for further details.
	\end{enumerate}
\end{remark}

\begin{remark} \label{rem:QR_and_Hessenberg_still_Hessenberg_plus_arithmetic_cost}
    If $M$ is a Hessenberg matrix, we can restrict the permutations used in the standard $PLU$-factorization algorithm to compute a $QR$-factorization $M=QR$ such that $Q$ is a Hessenberg matrix. Then $M_{\text{next}}:=RQ$ is the product of a Hessenberg matrix with an upper-triangular matrix, so is also a Hessenberg matrix.
    Moreover, at worse $2n$ row operations ($n$ row eliminations plus $n$ row transpositions) are needed to compute $Q$ and $R$ from $M$.
    Computing $M_{\text{next}}=RQ$ can then be done in $2n$ columns operations.
    Since row/column permutations do not require arithmetic operations (only memory allocations or pointer reassignment, depending on the implementation), the cost of one $QR$-round applied to a Hessenberg matrix
    is bounded by $n^2$ arithmetic operations. If we also compute an update $V \mapsto Q^{-1}V$ to a transformation matrix, the total cost is $2n^2$ arithmetic operations.
\end{remark}

We now come to the discussion of $p$-adic precision. There are many ways to represent a $p$-adic element  $a \in \Qp$ in a computer system \cite{precision_book}. We represent an element of $\Qp$ by a truncated series
\[
a = a_{-r}p^{-r} + \ldots + a_0 + pa_1 + a_2p^2  + \ldots + a_{N-1}p^{N-1} + O(p^N)
\]
where the $O(p^N)$ is the $p$-adic ball representing the uncertainty of the remaining digits. The \emph{relative precision} of $a$ is the quantity $N + r$, and the \emph{absolute precision} is the number $N$. In the terminology of \cite{precision_book}, we consider a system with the \emph{zealous} (i.e, \emph{interval}) implementation of arithmetic. The operations $-, +$ preserve the minimum of the absolute precision of the operands, and $\times, \div$ preserve the minimum relative precision of the operands. If $u \in \mbz_p^\times$, $a \in \mbz_p$, and $N \leq N'$, then we have that $(u+O(p^{N'}))(a + O(p^{N})) = ua + O(p^N)$. Multiplication by $p$ preserves the relative precision and increases the absolute precision by $1$. The worst operation when it comes to absolute $p$-adic precision is dividing a small number by $p$. For example, the expression
\[
\frac{(1 + p^{99} + O(p^{100})) - (1 + O(p^{100}))}{p^{100} + O(p^{200})} = p^{-1} + O(1)
\]
begins with $3$ numbers with an absolute and relative precision of at least $100$, and ends with a result where not even the constant term is known. Henceforth, by \emph{precision} we refer to the absolute precision.

\begin{definition} \label{def: matrix oh-notation}
	Let $A,B \in \M_n(\Zp)$ be matrices such that $a_{i,j} = b_{i,j} + O(p^{N_{i,j}})$. Then we write $A = B + O(p^N)$, where $N := \min_{i,j} N_{i,j}$.
\end{definition}
To refer to a matrix $A \in \M_n(\Qp)$ whose elements are known at an absolute precision at least $N$, we will simply write $A + O(p^N)$. 
The same absolute precision on every entry
is called a \textit{flat} precision.

\subsection{Orthogonality and the Bilinear Lemma}


In pnumerical linear algebra, we often need to bridge the gap between an approximate computation -- usually, where arithmetic is performed in the ring $\Zp/p^N\Zp$ -- and some information about the true solution to our problem over $\Zp$. For example, consider computing the kernel of the following matrix equation
\[
	Mx = 
	\bbm
		p^3 & 0 \\ 0 & 0  
	\ebm
	x
	= 0.
\]
Over $\Zp$, we see that this matrix plainly has rank $1$, and our kernel is given by $e_2$. However, the kernel of $M \otimes_{\Zp} \Zp/p^N\Zp$ will always be rank $2$ as a $\Zp/p^N\Zp$-module. Thus, it is helpful to understand the properties of $\ker_{\Zp} M \otimes_{\Zp} \Zp/p^N\Zp$ to best make sense of the approximate computations. This leads us to the concept of $p$-adic orthogonality as introduced in \cite{schikhof2006ultrametric}.

\begin{definition}
	A set $\{x_1, \ldots, x_r\} \subset \Qp^n$ is \emph{orthogonal} if for every $\lambda_1, \ldots, \lambda_r \in \Qp$ we have that
	\[
	\norm{\sum_{j=1}^r \lambda_j x_j} = \max\left\{\abs{\lambda_j} \norm{x_j} : 1 \leq j \leq r\right\}.
	\]
	We say $\{x_1, \ldots, x_r\}$ is \emph{orthonormal} if it is orthogonal and each $\norm{x_j} = 1$.
\end{definition}

\begin{definition}
	A submodule $V \subseteq \Zp^n$ is \emph{orthonormally generated} if it is generated by an orthonormal set. We also say that $V$ \emph{admits an orthogonal basis}.
\end{definition}

Note that a subset $\{x_1, \ldots, x_r\} \subset \Zp^n$ is orthonormal if and only if $r \leq n$ and there is a $U \in \GL_n(\Zp)$ such that $x_j = e_j U$ for all $1 \leq j \leq r$. Since any two bases of a free $\Zp$-module are related by a transformation in $\GL_n(\Zp)$, we obtain the following basis-free characterizations of the orthonormally generated criterion.

\begin{lemma}
	Let $V$ be a free $\Zp$-submodule of $\Zp^n$. 
	\begin{enumerate}[(a)]
		\item
			If $V$ admits an orthonormal basis, then every basis of $V$ is orthonormal. 
		\item
			We have that $V$ is orthonormally generated if and only if the cokernel of the inclusion $V \hookrightarrow \Zp^n$ is a free $\Zp$-module.
	\end{enumerate}
\end{lemma}

We additionally have a notion of orthogonal complement.

\begin{definition}
	Two submodules $U,V \subseteq \Zp^n$ are orthogonal if for some choice of bases $\{ u_i, i \in I \}$, $\{ v_j, {j \in J} \}$ the set $\{ u_i, i \in I \} \cup \{ v_j, j \in J \}$ is orthogonal. If $V$ and $U$ are both orthonormally generated and $\Zp^n = V \oplus U$, we say that $U$ is \emph{an orthogonal complement} to $V$ (and \textit{vice-versa}).
\end{definition}

Given a submodule $V \subseteq \Zp^n$ that is orthonormally generated, it is easy to construct an orthogonal complement. Writing a basis for $V$ as the rows of an $r \times n$ matrix $M$, we compute a singular value decomposition $M = Q \Sigma P$. Note that $\Sigma$ has unit entries on the diagonal, as $V$ is orthonormally generated, and that the first $r$ rows of $P$ generate $V$ as a submodule. Since $P \in \GL_n(\Zp)$, we see that the last $n-r$ rows of $P$ generate an orthonormal module orthogonal to $V$. That being said, the orthogonal complement of a non-trivial subspace is never unique.

\bigskip
A useful result to relate the results of our computations back to results over $\Zp$ is the Bilinear Lemma of Samuel-Zariski \cite[Chapter VIII, Section~7]{ZariskiSamuel}.

\begin{lemma}[Bilinear Lemma]
	Let $A$ be a ring, $\mathfrak{m}$ an ideal in $A$, and let $E,E',F$ be three $A$-modules. Assume that $F$ is a Hausdorff space for its $\mathfrak{m}$-topology and that $A$ is complete. Let $f\colon E \times E' \rightarrow F$ be a bilinear mapping, and denote by $\bar f\colon E/\mathfrak{m}E \times E'/\mathfrak{m}E' \rightarrow F / \mathfrak{m}F$ the canonically determined map. 
	
	If we are given $y \in F, \bar \alpha \in E/\mathfrak{m}E, \bar \alpha' \in E'/\mathfrak{m}E'$ such that $\bar f(\alpha, \alpha') = \bar y$ and $F/\mathfrak{m}F = f(\bar \alpha, E'/\mathfrak{m}E') + f(E/\mathfrak{m}E, \bar \alpha')$. Then there are lifts of $\alpha, \alpha'$ to $E,E'$ such that $y = f(\alpha, \alpha')$.
\end{lemma}

We can translate this directly to our situation.

\begin{lemma}[Bilinear Lemma, specialized]
	Let $\mathfrak{m}$ an ideal in $\Zp$. If we are given $y \in \Zp^n, \bar x \in (\Zp/\mathfrak{m}\Zp)^n, \bar M \in \M_n(\Zp/\mathfrak{m}\Zp)$ such that $\bar x$ has a unit coordinate and $\bar x \bar M = \bar y$. Then there is a lift $x \in \Zp^n$ of $\bar x$ and a lift $M \in \M_n(\Zp)$ of $\bar M$ such that $xM = y$.
\end{lemma}

\begin{proof}
	The hypotheses of the general Bilinear Lemma are readily checked.
\end{proof}

\noindent
Finally, we define the notion of orthogonality, orthonormal, and orthogonal complement for $(\Zp/p^N\Zp)^n$.

\begin{definition}
	Let $V$ be a submodule of $(\Zp/p^N\Zp)^n$. Then $V$ is \emph{orthonormally generated} if the cokernel of the inclusion $V \hookrightarrow (\Zp/p^N\Zp)^n$ is a free $(\Zp/p^N\Zp)$-module.
\end{definition}

\begin{definition}
	Two submodules $U,V \subseteq (\Zp/p^N\Zp)^n$ are orthogonal if for some choice of bases $\{ \bar u_i, i \in I \}$, $\{ \bar v_j, {j \in J} \}$ and any lifts $\{ u_i, i \in I \}$, $\{ v_j, {j \in J} \}$ to $\Zp^n$, the set $\{ u_i, i \in I \} \cup \{ v_j, j \in J \}$ is orthogonal. If $V$ and $U$ are both orthonormally generated and $(\Zp/p^N\Zp)^n = V \oplus U$, we say that $U$ is \emph{an orthogonal complement} to $V$ (and \textit{vice-versa}).
\end{definition}

\subsubsection{pNumerical ranks, kernels, and preimages} \label{sssec: pnumerical spaces}

In this section, we define the pnumerical rank, kernel, and inverse image. We also discuss how to compute such objects and how they relate to their exact counterparts for a matrix $M \in \M_n(\Zp)$. 

\begin{definition}
	The \emph{pnumerical rank of precision $O(p^{N})$} of $M \in \M_n(\Zp)$ is the number of singular values of $M$ of norm strictly bigger than $p^{-N}$ (\textit{i.e.} of valuation strictly smaller than $N$). 
\end{definition}

With a sufficient amount of precision, the pnumerical rank will be equal to the rank. Additionally, the strict $QR$-factorization will reveal the pnumerical rank of the original matrix as the number of non-zero pivots of $R$. If not enough precision is given, this cannot be guaranteed. The singular value decomposition always reveals the pnumerical rank.

\begin{example} \label{ex: QR rank needs precision}
	For the matrix
	\[
	M := \bbm
	p & 1 \\
	0 & p & 1 \\
	0 & 0 & p
	\ebm
	+ O(p^3)
	\]
	we see with $Q := 1, R := M$ that $M = QR$ is a strict $QR$-factorization. However, because of the low precision ($\vert \sigma_* (M) \vert \leq \vert p^3 \vert$), we can obtain another strict $QR$-factorization with
	\[
		Q' := \bbm
		1 & 0 & 0\\
		0 & 1 & 0 \\
		-p^2 & p & 1
		\ebm
		+ O(p^3)
		,
		\qquad
		R' := \bbm
		p & 1 \\
		0 & p & 1 \\
		0 & 0 & 0
		\ebm
		+ O(p^3).
	\]
	We see that the second strict $QR$-factorization reveals the pnumerical rank, and the first does not.
\end{example}

We now discuss pnumerical kernels and pnumerical inverse images.
\begin{definition}
	Let $N > 0$. The \emph{pnumerical kernel of precision $O(p^N)$} of $M \in \M_n(\Zp)$ is the maximal free $(\Zp/p^N\Zp)$-submodule of $(\Zp/p^N\Zp)^n$ annihilated by $M$. 
\end{definition}

\begin{definition} \label{defn:pnum_preimage}
	Let $N > 0$. The \emph{pnumerical preimage of precision $O(p^N)$} of a submodule $V \subseteq (\Zp/p^N\Zp)^n$ under $M \in \M_n(\Zp)$ is the the maximal free $(\Zp/p^N\Zp)$-submodule $U \subseteq (\Zp/p^N\Zp)^n$ such that $MU \subseteq V$. 
\end{definition}


The pnumerical kernel of $M$ is not generally the kernel of $M \pmod {p^N}$ as an endomorphism of $(\Zp/p^N\Zp)^n$. As expected, the pnumerical kernel is just the pnumerical preimage of $0$. Generally, if there is no risk of confusion we will forgo stating the ``of precision $O(p^N)$'' part of these terms.

\begin{lemma} \label{lem: dual module}
	Let $V$ be a submodule of $(\Zp/p^N\Zp)^n$. Then there exists a matrix $M \in \M_n(\Zp/p^N\Zp)$ such that the kernel of $M$ (as an endomorphism of $(\Zp/p^N\Zp)^n$) is exactly $V$. If $V$ is orthonormally generated of rank $b$, then $M$ has $n-b$ singular values of size $1$ and $b$ singular values of size $0$. 
\end{lemma}

\begin{proof}
	Note that $V$ is a finitely generated $\Zp$-module, so by the structure theorem for modules over a PID we have that there is an isomorphism $\varphi\colon V \longrightarrow \bigoplus_{j} (\Zp/p^j\Zp)^{m_j}$ with all but finitely many $m_j \in \mathbb{N} \cup \{0\}$ equal to $0$. We let $B$ be the finite subset of $V$ obtained by pulling back a set of generators for the direct summands of $\bigoplus_{j} (\Zp/p^j\Zp)^{m_j}$ under $\varphi$. Denote $b := \#B$.
	
	Let $X \in \M_n(\Zp/p^N\Zp)$ be the $n \times b$ matrix whose columns are the elements of $B$, and let $X = Q\Sigma P$ be a singular value decomposition. In particular, we have that the bottom $(n-b) \times b$ block of $\Sigma$ is $0$. We lift the entries of $\Sigma$ to $\Zp$ and construct the $n \times n$ matrix $M$ by
	\[
		M := 
		\bbm
		p^{N}\sigma_1^{-1} \\
		 & \ddots \\
		 & & p^{N}\sigma_b^{-1} \\
		 & & & 1 \\
		 & & & & \ddots \\
		 & & & & & 1
		\ebm
		Q^{-1}
		\in \M_n(\Zp/p^N\Zp).
	\]
	We see $MX = 0$. Thus, the kernel of $M$ is exactly $V$, so this completes the first part of the lemma. If $V$ is orthonormally generated, then our $b$ is also the rank of $V$ and each of the $\sigma_j = 1$, so the second part follows.
\end{proof}

\begin{proposition} \label{prop: kernels are orthonormal}
	Let $N \in \mathbb{N}$ and let $M \in \M_n(\Zp)$ be a matrix given at flat precision $O(p^N)$ and let $V \subseteq \Zp^n$ be a $\Zp$-submodule. Then:
	\begin{enumerate}[(a)]
		\item
		The pnumerical kernel of $M$ is 
		orthonormally generated. Furthermore, the pnumerical kernel contains the image of $\ker M$ under reduction modulo $p^N$.
		
		\item
		If $V$ is orthonormally generated, the pnumerical preimage of $V$ is orthonormally generated. Furthermore, the pnumerical preimage of $V$ contains the image of the preimage of $V$ under reduction modulo $p^N$.

	\end{enumerate}
\end{proposition}

\begin{proof}
	Both parts can be deduced by using the $p$-adic singular value decomposition (i.e, the Smith normal form). For part (a), we see the result from the singular value decomposition for $M$. For part~$(b)$, we apply Lemma~\ref{lem: dual module} to find a matrix $A$ such that $\ker A = V$. In particular, the pnumerical inverse image of $V$ under $M$ is the pnumerical kernel of  $AM$, so we can compute it via the singular value decomposition as before. The statements regarding reduction modulo $p^N$ are obvious.
\end{proof}

Proposition~\ref{prop: kernels are orthonormal} is optimal in the following sense. If $M$ is known at flat precision $O(p^N)$ and $\bar V$ is the pnumerical kernel of $M$, then there exists some $M' \in \M_n(\Zp)$ and $V \subseteq \Zp^n$ such that $M'V = 0$,  $M = M' + O(p^N)$, and $\bar V = V \otimes \Zp/p^N\Zp$. This is an immediate consequence of the Bilinear Lemma. The analogous statement is true for the pnumerical inverse image.


\subsection{The basic $QR$-algorithm}

Algorithm~\ref{algo: simple QR} below is the simple $QR$-algorithm given in \cite{Kulkarni2019} (Algorithm~2.19 \emph{loc.$\!$ cit.}). This version suffers from a number of drawbacks: the algorithm only converges linearly and cannot decompose any block with eigenvalues that are the same modulo $p$. The core idea to improve the algorithm is the classic strategy of concurrently updating the approximation to the eigenvalue and the matrix.

\begin{algorithm}[h]
		\caption{\texttt{simple\_QR\_Iteration}($M$, $\chi_{A,p}$)}
		\label{algo: simple QR}
		\begin{algorithmic}[1]
			\REQUIRE
			\ \\
			$M + O(p^N)$, an $n \times n$-matrix $in \M_n(\mbz_p)$. \\
			$\chi_{M,p}$, the characteristic polynomial of $M \pmod p$. \\[1ex]

			\ENSURE
			A (block) triangular form $T$ for $M$, and a matrix $V$ such that $MV = VT + O(p^N)$ (\textit{i.e.} $V$ is a change of basis matrix between $T$ and $M$). \\[1ex]

			\STATE
			Set $\lambda_1, \ldots, \lambda_\ell$ to be the roots of $\chi_{M,p}$ in $\Fp$, lifted to $\Zp$.
			\STATE
			Set $m_1, \ldots, m_\ell$ to be the multiplicities of the roots of $\chi_{A,p}$.
					
			\STATE
			Compute $B,V$ such that $MV = VB$ and $B$ is in Hessenberg form.
			
			\FOR {$i= 1, \ldots, \ell$}
			\FOR {$j= 1, \ldots , m_iN$}
			\STATE
			Factor $(B-\lambda_i I) = QR$
			\STATE
			Set $B := RQ + \lambda_i I$
			\STATE
			Set $V := Q^{-1}V$						
			\ENDFOR
			\ENDFOR
			\RETURN $B, V$
		\end{algorithmic}
\end{algorithm}

We point out a useful lemma of Wilkinson from \cite{Wilkinson1965}, which helps in analysing the diagonal elements of the various upper triangular factors encountered in the iteration.

\begin{lemma}[Wilkinson] \label{lem: Wilkinson lemma}
	Let $M \in \M_n(\Qp)$ be a matrix, let $s \geq 1$ be an integer, and let $(Q^{(1)}, R^{(1)})$, $\ldots$, $(Q^{(s)}, R^{(s)})$ be the $QR$-pairs for $s$ $QR$-rounds. Let
	\[
	M^{(s)} := R^{(s-1)}Q^{(s-1)}, \quad \mathcal{Q}^{(s)} := Q^{(1)} \cdots Q^{(s)}, \quad \mathcal{R}^{(s)} := R^{(s)} \cdots R^{(1)}.
	\]
	Then $M^s = \mathcal{Q}^{(s)}\mathcal{R}^{(s)}$ and $M^{(s+1)} = {\mathcal{Q}^{(s)}}^{-1} M \mathcal{Q}^{(s)}$.
\end{lemma}


We quote from \cite[Section~5]{Wilkinson1965} a brief summary of Wilkinson's argument to show why $QR$-iteration converges, in a simple case. We refer to Wilkinson's original article for the other cases. Let $M \in \M_n(\Zp)$ and assume that $M = XDX^{-1} = XDY$ for $D$ a diagonal matrix with $\lambda_j := D\ind{j,j}$ and $\abs{\lambda_1} > \ldots > \abs{\lambda_n} > 0$ and let $X \in \GL_n(\Zp)$. We will further assume that $X = L_XR_X$ and $Y = L_YR_Y$ for some $L_X, L_Y \in \GL_n(\Zp)$ unit lower-triangular matrices and $R_X, R_Y \in \GL_n(\Zp)$ upper-triangular.

Letting $(L^{(1)}, R^{(1)})$, $\ldots$, $(L^{(s)}, R^{(s)})$ be the $QR$-pairs for $s$ $QR$-rounds (with ${M = L^{(1)}R^{(1)}}$), we have
\[
	M^s = XD^sY = X(D^sL_YD^{-s})(D^sR_Y).
\]
We write $D^sL_YD^{-s} = I + F_s$, and we have
\[
	F_s\ind{i,j} = \begin{cases}
		L_Y\ind{i,j} \cdot \left( \frac{\lambda_i}{\lambda_j} \right)^s & \text{if } i > j \\
		0 & \text{if } i \leq j.
	\end{cases}
\]
By the assumption on the norms of the $\lambda_j$'s we have that $\lim_{s \rightarrow \infty} F_s \rightarrow 0$. We have
\begin{align*}
	XD^sY &= L_XR_X(1+F_s)D^sR_Y \\
	&= L_X(I + R_XF_sR_X^{-1})R_X D^s R_Y.
\end{align*}
Since $F_s \rightarrow 0$ under the iteration, for some sufficiently large $s$ we have that the $QR$-factorization of $(I + R_XF_sR_X^{-1})$ is of the form $(I+L')(I + R')$ with $L',R'$ lower (resp. upper) triangular and tending to $0$. In particular, by Wilkinson's Lemma
\[
	\mathcal{L}^{(s)}\mathcal{R}^{(s)} = \underbrace{L_X(I + L')} \underbrace{(I + R')(R_XD^sR_Y)}.
\]
The left factor is lower triangular and the right factor is upper triangular, so by the uniqueness of $LR$-decompositions of non-singular matrices over a domain, we have that $\mathcal{L}^{(s)} = L_X(I + L')$. But now with $M^{(s)}$ the $s$-th iterate of $M$ under the $QR$-iteration we have by Wilkinson's Lemma
\begin{align*}
	M^{(s)} &= (\mathcal{L}^{(s)})^{-1}M(\mathcal{L}^{(s)})  \\
	&= (\mathcal{L}^{(s)})^{-1} XDX^{-1} (\mathcal{L}^{(s)}) \\
	&= (I + L')^{-1}L_X^{-1}L_X R_X D R_X^{-1}L_X^{-1}L_X(I + L') \\
	&= (I + L')^{-1} R_X D R_X^{-1} (I + L').
\end{align*}
Because $\lim_{s \rightarrow \infty} (I+L') = I$, we see the $M^{(s)}$ converge to the upper triangular matrix $R_X D R_X^{-1}$.

\subsection{Problematic examples} \label{sec: troublesome examples}

Before proceeding with the rest of the article, we include examples that highlight some of the technical difficulties we need to be aware of in our proofs. First, we review an example from \cite{Kulkarni2019}.

\begin{example} \label{ex: charpoly needs more precision}
	Consider the matrix
	\[
	A := 
	\begin{bmatrix}
	p^3 & p^2 \\ 0 & -p^3
	\end{bmatrix} + O(p^6).
	\]
	The characteristic polynomial computed using capped precision arithmetic is $\chi_A + O(p^6) = T^2 + O(p^6)$. There is a precision loss in computing the roots of $f$, and the absolute error on the roots of $f$ cannot be better than $O(p^3)$. However, it is possible to know the characteristic polynomial of $A$ at a higher precision; keeping track of extra digits of precision, we have
	\begin{align*}
	\chi_A &= (p^3 + O(p^6) - T)(-p^3 + O(p^6) - T) - (p^2+O(p^6))(0 + O(p^6)) \\
	&= T^2 - (p^3 - p^3 + O(p^6))T + (p^6 + O(p^{9})) - (O(p^8)) \\
	&= T^2 - (0 + O(p^6))T + (p^6 + O(p^8)).
	\end{align*}
	With the extra digits of precision on the last coefficient of $\chi_A$, we can compute the roots of $\chi_A$ with an absolute error of $O(p^4)$. In particular, even when the input has flat precision, there are cases where the characteristic polynomial needs to be known at higher precision to obtain the best accuracy on the eigenpairs.
\end{example}

Next, we discuss topologically nilpotent matrices.
\begin{definition}
	We say a matrix $M \in \M_n(\Zp)$ is \emph{topologically nilpotent} if $\lim_{j \rightarrow \infty} \norm{M^j} = 0$. 
\end{definition}

For example, any matrix of the form
\[
\bbm
0 & \ldots & & 0 \\
1 & \\
& \ddots \\
& & 1 & 0
\ebm
+ pX, \quad X \in \M_n(\Zp)
\]
is topologically nilpotent. Topologically nilpotent matrices generally exhibit the worst-case scenario for the computation of the characteristic polynomial or iterative eigenvector algorithms  \cite{CRV2017characteristic, Kulkarni2019}. Practically, either more precision or more iterations are required to compute the generalized eigenspaces in these cases. Topologically nilpotent matrices are a particular examples of matrices $M \in \M_n(\Qp)$ such that $\abs{\lambda_1 - \lambda_2} < 1$ for some eigenvalues $\lambda_1, \lambda_2$ of $M$. In the archimedean case, the distances between eigenvalues of an input matrix is well-known to be related to the condition number of the eigenvalue/eigenvector problem.

\begin{example}[Topologically nilpotent matrices]
	Topologically nilpotent blocks present a worst case scenario for the convergence of our $QR$ method. Consider the $(n+1) \times (n+1)$ matrix
	\[
	\bbm
	1 & \\
	p & 0 & \ldots & 0 & p \\
	& 1 & \\
	& & \ddots \\
	& & & 1 & 0
	\ebm
	+ O(p^{N})
	\]
	After $n-1$ rounds (resp. $n$) $QR$-rounds (with shift $0$), we end up with the matrix
	\[
	\bbm
	1 & \\
	p & 0 & \ldots & & 0 & 1 \\
	& p & \\
	& & 1 \\
	& & & \ddots \\
	& & & & 1 & 0
	\ebm
	+ O(p^{N}), 
	\qquad
	(resp.) \quad 
	\bbm
	1 & \\
	p^2 & 0 & \ldots & & 0 & p \\
	& 1 & \\
	& & 1 \\
	& & & \ddots \\
	& & & & 1 & 0
	\ebm
	+ O(p^N).
	\]
	We see that the convergence of the $(2,1)$-entry to zero is hampered by the chain of subdiagonal $1$'s. If $\lambda_1, \lambda_2$ are distinct small eigenvalues and $\gcd(p,n) = 1$, we have $\abs{\lambda_1 - \lambda_2} = p^{-\frac{1}{n}}$, so we do not meet the criterion for quadratic convergence. Second, this example suggests even in optimal cases why we may need $n\log_2 N$ iterations for the $(2,1)$ entry to converge to zero modulo $p^N$; essentially, we can only guarantee that the size of this entry decreases within $n$ iterations.
\end{example}

\begin{example}[Disordered eigenvalues] \label{ex: disordered eigenvalues}
	Consider the matrix
	\[
	\bbm
	1 & 0 & 1 & 0 & 0 \\
	1 & 1 & 3 & 0 & 0 \\
	& p & p & p & 1 \\
	& & p^{10} & 1 & 2 \\
	& &  & p & p
	\ebm
	+ O(p^{N}).
	\]
	It is not immediately clear what the change of coordinates is to ensure that the matrix remains in Hessenberg form and for the backward orbit of $0 \pmod p$ to correspond to the last two (row) vectors. The transformation to convert this matrix to a {\thingname } appears to be difficult to compute.
\end{example}

There are several ways in which a matrix in $\M_n(\Zp)$ can fail to be diagonalized by a $\GL_n(\Zp)$ transformation. The first is that the matrix is not semi-simple, and the second is that the characteristic polynomial of $M$ may contain non-trivial irreducible factors. There is a third obstruction to diagonalizability whenever the singular values differ from the sizes of the eigenvalues. 

\begin{example}[Non-$\GL_n(\Zp)$-diagonalizable matrices.] \label{ex: non GLn(Zp) diagonalizable}
	Consider the topologically nilpotent matrix
	\[
		M = \bbm p & 1 \\ 0 & 0 \ebm.
	\]
	We see that $M$ is in Schur form and that the eigenvalues are $\{0,p\}$. It is impossible to diagonalize $M$ over $\GL_n(\Zp)$, as $\GL_n(\Zp)$ conjugation preserves the singular values, which in this case are $\{0,1\}$. 
	
\end{example}

\section{Computing {\sizesorted } forms and generalized $0$-eigenspaces} \label{sec: gze}

In this section, we study the connection between {\sizesorted } forms of a matrix and approximations to the generalized $0$-eigenspace. We first present a refinement of the Hodge-Newton decomposition from \cite[Theorem~4.3.11]{Kedlaya2010differential}. To begin, we give a variant of a classical result.

\begin{lemma} \label{lem: triangularization lemma}
	Let $f \in \Zp[t]$ be a polynomial, let $M \in \M_n(\Zp)$, let $V := \lker f(M)$, and let $r := \rank V$. Then $V$ is orthonormally generated. Moreover, there exists a $U \in \GL_n(\Zp)$ such that $VU^{-1} = \gen{e_{n-r+1}, \ldots, e_n}$. In particular,
	\[
	UMU^{-1} = \bbm A & C \\ 0 & B \ebm.
	\]
\end{lemma}

\begin{proof}
	Since $V$ is a kernel, it is orthonormally generated and admits an orthogonal complement $V^\perp$. Representing a basis $b_1, \ldots, b_{n-r}$ for $V^\perp$ and a basis $b_{n-r+1}, \ldots, b_{n}$ for $V$ as row vectors we construct
	\[
	U := \bbm b_1^T & \cdots & b_n^T \ebm^T.
	\]
	By orthogonality we have $U \in \GL_n(\Zp)$ and by definition $U$ sends $\gen{e_{n-r+1}, \ldots, e_n}$ to $V$. Finally, $M$ commutes with $f(M)$, so $V$ is an invariant subspace for $M$. In particular, $VM \subseteq V$. By the definition of $U$ we have
	$
	UMU^{-1} = \bbm A & C \\ 0 & B \ebm
	$
	as required.
\end{proof}

The lemma above allows us to show that a factorization of $\chi_M$ indicates that $M$ can be put into a matching block triangular form by a $\GL_n(\Zp)$ transformation. We can now prove the first theorem from the introduction.

\begin{theorem} \label{thm: Schur forms via GLn(Zp) transforms}
	Let $M \in \M_n(\Zp)$ and let $\chi_M = f_1 \cdots f_r$ be a factorization in $\Zp[t]$ where the factors are pairwise coprime in $\Qp[t]$.  Then there exists a $U \in \GL_n(\Zp)$ such that $UMU^{-1}$ is block-triangular with $r$ blocks; the $j$-th block accounts for the eigenvalues $\lambda$ such that $f_j(\lambda) = 0$.
\end{theorem}

\begin{proof}
	Using Lemma~\ref{lem: triangularization lemma} with the polynomial $f_r$, we can find a $U \in \GL_n(\Zp)$ such that
	\[
	UMU^{-1} = \bbm A & C \\ 0 & B \ebm.
	\]
	Since the $f_j$ are pairwise coprime, we have $\chi_B = f_r$. The result follows from an inductive argument.
\end{proof}


Even though a $\GL_n(\Zp)$ transform can be found to put a matrix into a block Schur form, this does not mean a $\GL_n(\Zp)$ matrix can be found that block diagonalizes the matrix. See Example~\ref{ex: non GLn(Zp) diagonalizable}. If $f \in \Zp[t]$ is a polynomial whose roots have valuations $\{\nu_1, \ldots, \nu_R\}$, there is a factorization $f = f_1 \cdots f_r$ where the roots of each $f_j$ have valuation $\nu_j$
(see \cite[Section 2.2]{Kedlaya2010differential}, \cites{MontesProject, GuardiaNartPauli2012, CRV2016slopefactorization} for more details on the factorization of $p$-adic polynomials, slope
factorization and how to compute them).
Thus we obtain:

\begin{corollary}[Newton decomposition] \label{cor: Newton decomposition}
	Let $M \in \M_n(\Zp)$ and let $\nu_1 \leq \ldots \leq \nu_r$ be the distinct valuations of the eigenvalues of $M$. Then there exists a $U \in \GL_n(\Zp)$ such that $UMU^{-1}$ is block-triangular with $r$ diagonal blocks; the $j$-th block accounts exactly for the eigenvalues of valuation $\nu_j$.
\end{corollary}

To compute a sorted matrix, we can use the standard algorithm to compute a block Schur form for a matrix over $\mathbb{F}_p$. We state this as Algorithm~\ref{algo: sorted form algorithm}. Note that a {\sizesorted } form is a $1$-digit of precision approximation to the Newton decomposition from Corollary~\ref{cor: Newton decomposition}.

\begin{algorithm}[h]
	\caption{\texttt{{sorted\_form}}}
	\label{algo: sorted form algorithm}
	\begin{algorithmic}[1]
		\REQUIRE
		\ \\
		An $n\times n$ matrix $M$ known at precision $O(p^N)$. \\
		
		\ENSURE
		A sorted form $M'$ for $M$ and a matrix $U \in \GL_n(\Zp)$ such that $M' = UMU^{-1}$. \\[1ex]
		
		\STATE
		Set $\bar M := M \pmod p$.

		\STATE
		Compute a block Schur form for $\bar M$ with change of basis matrix $U \in \GL_n(\mathbb{F}_p)$
		\STATE
		Lift $U$ to $\GL_n(\Zp)$
		\STATE
		Set $M' := UMU^{-1}$
		\RETURN $M'$, $U$
	\end{algorithmic}
\end{algorithm}

This algorithm is sufficient for our purpose of computing the block Schur form. We also see that there is a connection between computing the generalized $0$-eigenspace at $1$ digit of precision and the computation of a {\sizesorted } form of a matrix. Consequently, sorted matrices necessarily have a non-trivial factorization of their characteristic polynomials. 

\begin{lemma} \label{lem: chi factors}
	Let $M := \bbm A & C \\ E & B \ebm$ be a {\sizesorted } matrix and let $\epsilon$ be a positive power of $p$ such that $E \equiv 0 \pmod \epsilon$. Then there is a factorization $\chi_M = \chi_{\bigeig} \chi_{\smeig}$ in $\Zp[t]$ such that $\chi_{\bigeig} \equiv \chi_A \pmod \epsilon$ and $\chi_{\smeig} \equiv \chi_B \pmod \epsilon$. Moreover, the factorization $\chi_M \equiv \chi_A \chi_B \pmod \epsilon$ into monic polynomials is unique in $(\Zp/\epsilon \Zp)[t]$.
\end{lemma}

\begin{proof}
	Note that $\chi_M \equiv \chi_A \chi_B \pmod \epsilon$, and in particular $\chi_M \equiv \chi_A \chi_B \pmod p$. Writing $\chi_{A,p}, \chi_{B,p}$ for the reductions of $\chi_A, \chi_B$ modulo $p$ (respectively), we have $\gcd(\chi_{A,p}, \chi_{B,p}) = 1 \in \Zp/p\Zp$. By Hensel's lemma, we have the factorization $\chi_M = \chi_{\bigeig} \cdot  \chi_{\smeig}$ in $\Zp[t]$, and moreover, if $\chi_M \equiv F G \pmod{\epsilon}$ is a factorization such that $F \equiv \chi_{\bigeig} \pmod p$ and $G \equiv \chi_{\smeig} \pmod{p}$, then $F \equiv \chi_{\bigeig} \pmod{\epsilon}$ and $G \equiv \chi_{\smeig} \pmod{\epsilon}$. Thus, we see that $\chi_{\smeig} \equiv \chi_B \pmod \epsilon$. 
\end{proof}

The remainder of this section is devoted to computing the generalized $0$-eigenspace of a matrix. This offers two possible benefits. First, it allows us to repair the classical algorithm to deal with cases such as Example~\ref{ex: insufficient precision causes GZE error}. Secondly, we can potentially set up the iterative algorithms to block-triangularize topologically nilpotent matrices.

\subsection{Computing the generalized $0$-eigenspace: Problematic examples}

In this section we give some examples that demonstrate the difficulty of computing the generalized $0$-eigenspace.

\begin{example} \label{ex: insufficient precision causes GZE error}
	Let $A := \bbm p^2 \\  & 0 & 1 \\ & 0 & 0 \ebm + O(p^4)$. The schoolbook method to compute the (right sided) generalized $0$-eigenspace is to compute $\ker A^2$. Unfortunately, we see that $A^2 \equiv 0 \pmod {p^4}$, and in this case we do not compute the generalized $0$-eigenspace correctly.	
\end{example}

\begin{example}
	In infinite precision, another way to compute the generalized $0$-eigenspace is to iteratively solve $Ax = b$, starting with $b=0$. The corresponding calculation in finite precision is a little delicate. Consider the matrix
	\[
	A := 
	\bbm
	p \\ & 0 & 0 & 0 \\ & 1 & 0 & 0 \\ & 0 & p & 0 
	\ebm
	\in \M_4(\Zp).
	\]
	We see that the right kernel of $A$ is generated by $e_4$, and that the generalized $0$-eigenspace is $\gen{e_2, e_3, e_4}$. Unfortunately, in this case, the solutions in $\Zp^4$ to $Ax = e_4$ are of the form $x = p^{-1} e_3 + u e_4$, where $u \in \Zp$. Working with $4$-digits of precision, the $\Zp/p^4\Zp$-submodule of elements such that $\bar Ax \in \gen{\bar e_4}$ is $\gen{p^3\bar e_1, \bar e_3, \bar e_4}$. This example indicates we need to be careful about what we mean by the ``backward orbit of $0 \pmod{p^N}$'' and motivates the definition of the \emph{pnumerical preimage of precision $O(p^N)$} in Definition~\ref{defn:pnum_preimage}.
\end{example}

\subsection{The generalized $0$-eigenspace algorithm}	

We give an algorithm (Algorithm~\ref{algo: gze}) to compute the generalized $0$-eigenspace of a matrix $M$ given at finite precision.

\begin{algorithm}[h]
	\caption{Generalized 0-eigenspace (abbreviated to \texttt{GZE})}
	\label{algo: gze}
	\begin{algorithmic}[1]
		\REQUIRE
		\ \\
		An $n\times n$ matrix $M$ known at precision $O(p^N)$. \\
		
		An s.v.d factorization $M = Q\Sigma P$.
		
		\ENSURE
		A matrix whose rows form a basis of a numerical approximation of the generalized left $0$-eigenspace of $M$. \\[1ex]
		
		
		\STATE
		Set $K :=\left[e_{r+1},\dots,e_n \right] \in \Zp^{(n-r) \times n}$ with $r$ the pnumerical rank of $M$ at precision $O(p^N)$
		to be a matrix representing the pnumerical left kernel of $\Sigma$
		
		\IF {$K = \gen{0}$} \label{step: gze: exit condition}
		\RETURN $\emptyset$
		\ENDIF
		
		\STATE \label{step: gze: kernel}
		Set $V := KQ^{-1}$. Set $\delta := n-r$
		
		\STATE \label{step: gze: square subblock select}
		Set $B$ to be a $\delta \times \delta$ square sub-block of $V$ such that $B \in \GL_\delta(\Zp)$
		
		\STATE \label{step: gze: rref}
		Set $W := B^{-1}V$
		
		\STATE
		Set $J$ to be the set of the  indices of the pivot columns in $W.$
		
		\STATE \label{step: gze: column eliminate}
		Eliminate columns of $M$ using pivots from $W$: call this $M'$ \\
		
		i.e, compute $X \in \M_n(\Zp)$ such that for any $j \in J$, we have $(M - XW)\ind{\bullet, j} = 0$ \\[1ex]
		
		\STATE \label{step: gze: quotient}
		Delete the columns and rows in $M'$ indexed by $J$: call this $M''$
		
		
		\STATE
		Compute $M'' = Q''\Sigma''P''$ an s.v.d. decomposition
		
		\STATE
		Set $V_{new} := \texttt{GZE}(M'' = Q''\Sigma''P'')$
		
		\STATE
		Set $\widetilde{V_{new}}$, obtained from $V_{new}$ as a matrix with $n$ columns, with those of index in $J$ being $0$
		
		\RETURN $\left[\frac{\widetilde{V_{new}}}{V}\right]$
		
	\end{algorithmic}
\end{algorithm}

The underlying reason that this algorithm computes the correct answer is that after the truncation in step~\ref{step: gze: quotient}, we have that $V$ is the kernel of $M$ and that we constructed the operator $\overline M \colon \Zp^n/V \rightarrow \Zp^n/V$ up to a change of basis. We then use the fact that $\texttt{GZE}(M) \cong \texttt{GZE}(\overline M) \oplus V$. More precisely, one chooses an orthogonal complement $V^\perp$ to $V$ inside $\Zp^n$ and computes an operator $M'$ such that $M'$ stabilizes $V^\perp$ and the image of $M - M'$ is contained in $V$. In this case, $\texttt{GZE}(M) = \texttt{GZE}(M'|_{V^\perp}^{V^\perp}) \oplus V$.
%
Note that an orthogonal complement to $V$ in $\Zp$ is given by $\gen{ e_i : i \not \in J}$. This is easily seen from the fact that the pivots of $W$ occur in the columns indexed by $J$.

\begin{lemma}
	The matrix $M''$ computed on step~\ref{step: gze: quotient} represents $\overline M \colon \Zp^n/V \rightarrow \Zp^n/V$.
\end{lemma}

\begin{proof}
	First, we fix the basis $\gen{e_i : i \not \in J}$ for the choice of orthogonal complement. 
	Note that we have the equation $M' := M - XW$.
	The (left) image of $M'$ is contained in $\gen{e_i : i \not \in J}$. In particular, $M'$ defines an endomorphism of the subspace $\gen{e_i : i \not \in J}$. The explicit matrix describing this endomorphism on $\Zp^n/V$ with respect to the chosen basis is obtained from $M'$ by deleting the columns indexed by $J$. This is exactly the matrix $M''$.
\end{proof}

\begin{proposition}
	Let $M \in \M_n(\Zp)$ be a matrix given at flat precision $O(p^N)$ and let $V_0$ be the generalized left $0$-eigenspace of $M$. Then Algorithm \ref{algo: gze}
	computes an approximation at precision $O(p^N)$ of $V_0$, in $O(n^3 \dim(V_0))$ arithmetic operations at precision $O(p^N)$.
\end{proposition}
\begin{proof}
	From the previous lemma and discussion, it is clear that Algorithm~\ref{algo: gze} is correct
	when performing computations at infinite precision. Next, note that the rows of the matrix $V$ computed in step~\ref{step: gze: kernel} are orthonormal, as the rows of $V$ generate the kernel of $M$ as a morphism of $\Zp$-modules. Consequently, no divisions by $p$ are needed to compute the reduced row echelon form of $V$. 
	In the elimination on step~\ref{step: gze: column eliminate}, the pivot entries of $W$ are units, so no divisions by $p$ are needed to perform the eliminations. Since $J$ indexes both the pivots of $V$ and a collection of identically $0$ columns in $\widetilde{V_{new}}$, we see that the rows of $V$ and $\widetilde{V_{new}}$ are orthonormal. \\

	Let $\wtilde{V} := \left[\frac{\widetilde{V_{new}}}{V}\right]$ and let $\delta := \dim(V_0)$. By definition, we see that $\wtilde{V}M^\delta = 0 + O(p^N)$, so the module generated by the rows of $\wtilde{V} \pmod {p^N}$ is contained in the pnumerical kernel of $M^\delta$. On the other hand, the rows of $\wtilde{V}$ are orthonormal, so it is easy to see by the exit condition in step~\ref{step: gze: exit condition} that the rows of $\wtilde{V} \pmod{p^N}$ generate the pnumerical kernel of $M^\delta$. By Proposition \ref{prop: kernels are orthonormal} we see that the rows of $\wtilde V$ generate $V_0 \pmod {p^N}$.
	
	Finally, we comment on the computational complexity. The s.v.d computation and the eliminations in step~\ref{step: gze: column eliminate} can both be done with $O(n^3)$ arithmetic operations. Steps~\ref{step: gze: square subblock select} and~\ref{step: gze: rref} can be combined and done with $O(n^3)$ arithmetic operations using the $QR$-decomposition with column pivoting. The number of recursive calls is at most $\dim V_0$. In total, we perform $O(n^3 \dim(V_0))$ arithmetic operations.
\end{proof}

\begin{remark}
	The repeated computation of the singular value decomposition in Algorithm~\ref{algo: gze} means it is not efficient. Instead of using an s.v.d. decomposition, we can use a $QR$-decomposition. The advantage is that the $QR$-decomposition for $M''$ can be easily obtained from the $QR$-decomposition for $M$; since $M''$ is a $\rank(\ker M)$-update of $M$ followed by row/column deletion updates, we can use the $QR$-update algorithm of \cite[Section 6.5.1]{MatrixComputationsBook}, with Givens rotation replaced by $\GL_2(\Zp)$-elimination. The $QR$-update only requires $O(n^2)$ arithmetic operations when the kernel has rank $1$. However, the $QR$-decomposition is only rank revealing given sufficient precision (see Example~\ref{ex: QR rank needs precision}). We do not presently know how much precision is needed for this modification to work correctly.
\end{remark}


\section{The improved $QR$-iteration} \label{sec: technical improvements}

Our proof of super-linear convergence in the $QR$-iteration depends on being able to convert the matrix to a \thingname. First, we give the standard Hessenberg algorithm for reference. 

\begin{algorithm}[h]
	\caption{\texttt{standard\_hessenberg}}
	\label{algo: standard hessenberg}
	\begin{algorithmic}[1]
		\REQUIRE
		\ \\
		$M + O(p^N)$, an $n \times n$-matrix over $\mbz_p$. \\
	
		\ENSURE
		A Hessenberg form $H$ for $M$ and a $U \in \GL_n(\Zp)$ such that $HU = UM + O(p^N)$

		\STATE
		Set $U := I$
		
		\FOR {$j = 1, \ldots, n-1$}
			\STATE
			Find the minimal $i \in \{j+1, \ldots, n\}$ such that $\abs{M\ind{i,j}}$ is maximal
					
			\STATE
			Permute row $j+1$ and row $i$ in $M$. Permute row $j+1$ and row $i$ in $U$
			
			\IF {$M\ind{j+1,j} \neq 0$}
				\STATE
				Set $U\ind{i,j} := -M\ind{i,j} / M\ind{j+1,j}$ for $j+2 \leq i \leq n$ 
				
				\STATE
				Compute $UM$ by using row $j+1$ to eliminate each $M\ind{i,j}$ for $j+2 \leq i \leq n$ 
						
				\STATE
				Compute $MU^{-1}$ by applying column operations
			\ENDIF
		
		\ENDFOR
		
		\RETURN
		$M, U$
	\end{algorithmic}
\end{algorithm}

For attempting to compute a \thingname, we make two modifications to the standard Hessenberg algorithm. First, we start from the bottom and proceed upward rather than starting from the left and proceeding right. Secondly, we restrict the set of permutations in step~3 so that the sorted form is preserved. The reason we start from the bottom row in Procedure~\ref{proc: attempt sorted hessenberg} is that \emph{the procedure is guaranteed to produce a {\unshiftedthingname } if $b=1$}. This is because the condition in step~\ref{step: check if valid permutation} is vacuously false.

\begin{procedure}[h]
	\caption{\texttt{attempt\_sorted\_hessenberg}}
	\label{proc: attempt sorted hessenberg}
	\begin{algorithmic}[1]
		\REQUIRE
		\ \\
		$M + O(p^N)$, an $n \times n$ {\sizesorted } matrix over $\mbz_p$. \\
		
		The block sizes $(a,b)$ for $M$.
		
		\ENSURE
		A sorted Hessenberg form $H$ for $M$ and a $U \in \GL_n(\Zp)$ such that $HU = UM + O(p^N)$
		
		\STATE
		Set $U := I$
		
		\FOR {$i = n, \ldots, 2$}
		\STATE
		Find the maximal $j \in {1, \ldots, i-1}$ such that $\abs{M\ind{i,j}}$ is maximal
		
		\IF {$j \leq a$ and $a < i-1$} \label{step: check if valid permutation}
			\RETURN \texttt{Fail}, $M$, $U$
		\ENDIF
		
		\STATE
		Permute column $i-1$ and column $j$ in $M$. Permute column $i-1$ and column $j$ in $U$

		\IF {$M\ind{i,i-1} \neq 0$}			
			\STATE
			Set $U\ind{i,j} := M\ind{i,j} / M\ind{i,i-1}$ for $1 \leq j \leq i-2$
			
			\STATE
			Compute $MU^{-1}$ by using column $i-1$ to eliminate each $M\ind{i,j}$ for $1 \leq j \leq i-2$
			
			\STATE
			Compute $UM$ by applying row operations
		\ENDIF
		
		\ENDFOR
		
		\RETURN \texttt{Success}, $M, U$
	\end{algorithmic}
\end{procedure}

\subsection{Super-linear separation}

In the case that $M := [A; \epsilon, B]$ is a {\thingname }, the separating entry $\epsilon$ will deflate superlinearly to $0$ in the $QR$-iteration. We organize the proof of this statement into a sequence of three results.

\begin{lemma} \label{lem: one step QR}
    Let $M := [A;\epsilon, B]$ be a {\thingname } with block sizes $(n_A, *)$, and let $\mu \in \epsilon \cdot \Zp$. Write $M - \mu I = Q_MR_M$ and $B - \mu I = Q_B R_B$. Then $Q_M$ is block upper triangular modulo~$\epsilon$. Additionally, with
    \[
        M' := R_MQ_M + \mu I =: \blockHess{A'}{\epsilon'}{B'}, \quad \alpha := R_M\ind{n_A+1, n_A+1},
    \]
    we have that $\abs{\epsilon'} = \abs{\epsilon} \cdot \abs{\alpha}$ and $\abs{\alpha} \leq \max\{\abs{\epsilon}, \abs{R_B\ind{1,1}} \}$. 
\end{lemma}

\begin{proof}
    Write $A - \mu I = Q_A R_A$. Then
    \[
        (Q_A \oplus I)^{-1} (M - \mu I) = \blockHess{R_A}{\epsilon}{B - \mu I}.
    \]
    Let $r=R_A\ind{n_A,n_A} \in \Zp.$ Note that $ \abs{r}=\abs{R_A\ind{n_A,n_A}} \geq \abs{\sigma_*(R_A)} = 1$ since $\mu$ is small and we have assumed that all of the small eigenvalues correspond to the block $B$. So $\abs{r}=1.$

    The next operation in computing $M - \mu I = Q_M R_M$ is the elimination of the $\epsilon$ entry. The elementary row matrix for this step is $E := [{I_A;-r^{-1}\epsilon},I_B]$, and the resulting intermediate matrix is
    \[
        E\cdot (Q_A \oplus I)^{-1} \cdot (M - \mu I) =: [R_A; 0, C].
    \]
    Writing $C = Q_{C} R_{C}$ and $M - \mu I = Q_MR_M$, we have that
    \[
    	Q_M = \blockHess{Q_A}{-r^{-1} \epsilon}{Q_{C}},
    	\quad
    	R_M = \blockHess{R_A}{0}{R_{C}}.
    \]
    As $R_{C}\ind{1, 1} = \alpha$ and $M' = R_M Q_M + \mu I=[A'; \epsilon',B']$, then by direct calculation $\abs{\epsilon'} = \abs{\alpha} \cdot \abs{\epsilon}$.
    On the other hand, as $C \equiv B \mod \epsilon$ (since $\mu \equiv 0 \mod \epsilon$), we get that $\abs{\alpha} \leq \max\{\abs{\epsilon}, \abs{R_B\ind{1,1}} \}$, which concludes the proof.
\end{proof}

\begin{corollary} \label{cor: m steps QR}
	Let $M := [A; \epsilon, B]$ be a {\thingname } with block sizes $(*, m)$ such that $\norm{\chi_{\smeig} - t^m} \leq \abs{\epsilon}$. Then after $m$ $QR$-rounds, we obtain a matrix $M' := [A'; \epsilon', B']$ with
			$\abs{\epsilon'} \leq \abs{\epsilon}^{2}$.
\end{corollary}

\begin{proof}
	Let $\chi_M = \chi_{\bigeig} \chi_{\smeig}$. By Lemma~\ref{lem: chi factors} we have $\norm{ \chi_B - \chi_{\smeig}} \leq \abs{\epsilon}$ and by assumption, $\norm{\chi_{\smeig} - t^m} \leq \abs{\epsilon}$.
	Applying the Cayley-Hamilton theorem we then obtain:
	\[
		-B^m \equiv \chi_{B}(B) - B^m \equiv \chi_{\smeig}(B)-B^m  \equiv 0 \pmod \epsilon.
	\]
	Let $(Q_M^{(1)}, R_M^{(1)}), \ldots, (Q_M^{(m)}, R_M^{(m)})$ be the $QR$-pairs for the $m$ $QR$-rounds, define $R_A^{(j)}, R_B^{(j)}$ by $[R_A; 0, R_B] := R_M^{(j)}$. Let $\delta^{(j)} := R_B^{(j)}\ind{1,1}$ for each $1 \leq j \leq m$ and let 
	\begin{gather*}
		\mathcal{Q}_M^{(m)} := Q_M^{(1)} \ldots Q_M^{(m)}, \qquad
		\mathcal{R}_M^{(m)} := R_M^{(1)} \ldots R_M^{(m)},  \qquad
		\mathcal{R}_B^{(m)} := R_B^{(1)} \ldots R_B^{(m)}.
	\end{gather*}
	From Wilkinson's Lemma, $\mathcal{Q}^{(m)} \mathcal{R}^{(m)} = M^m$ and $M^m \equiv [A^m; 0, B^m] \pmod \epsilon$.
	As the $R$-factors are upper triangular, we have
	$
		\abs{\prod_{j=1}^m \delta^{(j)} } = \abs{ \mathcal{R}_B^{(m)}\!\ind{1,1}} \leq \abs{\epsilon}.
	$
	By applying Lemma~\ref{lem: one step QR} to all of the $QR$-rounds, we have either $\abs{\epsilon'} \leq \abs{\epsilon}^2$ or 
	\begin{equation*}
		\abs{\epsilon'} \leq 
		\left( 
		\left( 
		\left( \abs{\epsilon} \cdot \abs{\delta^{(1)}} \right) 
		\cdot  \abs{\delta^{(2)}} \right) 
		\ldots \right) 
		\cdot \abs{\delta^{(m)}} \leq \abs{\epsilon}^2. \tag*{} \qedhere
	\end{equation*}
\end{proof}

In the proof of Corollary~\ref{cor: m steps QR}, we only needed that $\norm{B^m e_1} \leq \abs{\epsilon}$. Eran Assaf pointed out to us that we can compute $B^me_1$ in $\mathcal{M}(m)\cdot \log_2m$ operations, and efficiently forecast whether $m$ $QR$-rounds will decrease the size of $\epsilon$ to $\abs{\epsilon}^2$ -- here $\mathcal{M}(m)$ denotes the number of operations needed to multiply two $m \times m$ matrices.

\begin{corollary} \label{cor: convergence when all eigenvalues small}
	Let $M := [A;\epsilon, B]$ be a {\thingname } with block sizes $(*,m)$ and $\abs{\epsilon} < 1$. Let $1 \leq \gamma \leq -\log_p \norm{\chi_{\smeig} - t^m}$ be a real value. Then after $(m \lceil \log_2  \log_p (\gamma) \rceil)$ $QR$-rounds,
	we obtain a {\thingname } $M' := [A'; \epsilon', B']$ with $\abs{\epsilon'} \leq p^{-\gamma}$. 
\end{corollary}

\begin{proof}
    Straightforward induction. 
\end{proof}

\subsection{Trace shifting}

We show how to choose shifts $\mu$ such that $[A - \mu I; \epsilon, B - \mu I]$ satisfies the condition on the size of the small characteristic polynomial, or if it does not, we can prove that two clusters of small eigenvalues can be separated modulo $\epsilon$. 

\begin{proposition} \label{prop: convergence speed determined by eigenvalues}
	Let $M := [A; \epsilon, B]$ be a {\unshiftedthingname } with block sizes $(*,m)$ and let $\mu := \frac{1}{m}\trace(B)$. Factor $\chi_M(t) = \chi_{\bigeig} \chi_{\smeig}$ (with $\chi_A$, $\chi_B$ equal
	to $\chi_{\bigeig},$  $\chi_{\smeig} \mod \epsilon$, respectively). 
	If for all pairs of distinct roots $\lambda_1, \lambda_2$ of $\chi_{\smeig}$, we have $\abs{\lambda_1 - \lambda_2} \leq \abs{\epsilon}$, then $\norm{\chi_{B}(t - \mu) - t^m} \leq \abs{\epsilon}$.  
	By contraposition, if $\norm{\chi_{B}(t - \mu) - t^m} > \abs{\epsilon}$, then
	there are some distinct roots $\lambda_1, \lambda_2$ of $\chi_{\smeig}$, such that $\abs{\lambda_1 - \lambda_2} > \abs{\epsilon}$.
\end{proposition}

\begin{proof}
    Let $K$ be the field of definition of the eigenvalues of $\chi_M$ with ring of integers $\mathcal{O}_K$. 
    Assume that for all pairs of distinct roots $\lambda_1, \lambda_2 \in O_K$ of $\chi_{\smeig}$, we have $\abs{\lambda_1 - \lambda_2} \leq \abs{\epsilon}$, \textit{i.e.} $\lambda_1 \equiv \lambda_2 \mod \epsilon.$
    By Lemma~\ref{lem: chi factors}, we have $\chi_B=\chi_{\smeig} \mod \epsilon$, so
    $\mu=\frac{1}{m}\trace(B)=\lambda_1 \mod \epsilon$. 
    We compute that:
    \begin{align*}
        \chi_{\smeig} (t-\mu)  &\equiv \prod_i \left( t-\lambda_i-\mu \right), \\
                        &\equiv t^m \mod \epsilon.
    \end{align*}
    As $\chi_B=\chi_{\smeig}\mod \epsilon,$ we can conclude that $\norm{\chi_{B}(t - \mu) - t^m} \leq \abs{\epsilon}$. 
\end{proof}

When $p \mid m$, there is a potential ambiguity in choosing the last digits of $\mu$. However, since only finding the common leading digits of the eigenvalues is necessary, we may make some arbitrary choice and convergence will be unaffected beyond the possibility of accidentally choosing a better shift than expected. \emph{In the specific (very common) case that $m=1$, we will always choose a good shift and the precision of $\epsilon$ will at least double at every step.} To clarify what we mean by common, see Remark~\ref{rem: most meaning}. 

Based on various experiments, the condition that $\norm{\chi_{\smeig} - t^m} \leq \abs{\epsilon}$ is genuinely necessary to ensure quadratic convergence.
We remark that the converse of Proposition~\ref{prop: convergence speed determined by eigenvalues} is false; consider
\[
	M := \left[ 1; p^2, \bbm p & 0 \\ 0 & -p \\ \ebm+O(p^2) \right] \quad p \neq 2.
\]
We have with $\epsilon := p^2$ that $\chi_B(t) \equiv (t-p)(t+p) \equiv t^2+O(p^2)$, but $(-p) \not \equiv p \pmod p^2$.

\begin{proposition} \label{prop: precision doubling proposition}
	Let $M := [A; \epsilon, B]$ be a \thingname, let $m=n_B$ and let $\lambda_1, \ldots, \lambda_m$ be the small eigenvalues of $M$. Let $\mu := \frac{1}{m}\trace(B)$. If $ \eta := \max_{i,j} \abs{\lambda_i - \lambda_j} \leq \abs{\epsilon}$, then after $m$ $QR$-rounds with shift $\mu$ we 
	obtain a size-sorted Hessenberg matrix $[A_{\mathrm{next}}; \epsilon_{\mathrm{next}}, B_{\mathrm{next}}]$ such that $\abs{\epsilon_{\mathrm{next}}} \leq \abs{\epsilon^2}$. 
	After at most $(m \lceil \log_2(-\log_p\eta) \rceil)$ rounds, the obtained $[A_{\mathrm{next}}; \epsilon_{\mathrm{next}}, B_{\mathrm{next}}]$ is such that $\vert \epsilon_{\mathrm{next}} \vert \leq \eta.  $
\end{proposition}

\begin{proof}
	The result follows from Proposition~\ref{prop: convergence speed determined by eigenvalues}, Corollary~\ref{cor: m steps QR}, and Corollary~\ref{cor: convergence when all eigenvalues small}. 
\end{proof}


\subsection{Further properties of the $QR$-iteration} \label{sec: future directions}

In this subsection, we prove some further results about the $QR$-iteration. This section is not necessary to implement our main algorithm, but is intended to explain some patterns we have observed in computing several examples. Some heuristics are supported by these results.

Separating eigenvalues would be useful to continue converging quickly. The only way we presently are aware of doing this is to compute some approximation of the characteristic polynomial. We have already seen that low precision approximations, such as $\chi_M \pmod p$, provide a mean to separate the eigenvalues. We explain how to efficiently approximate some factor of the characteristic polynomial during a $QR$-iteration. Unfortunately, it is possible that this approximation is not sufficient to separate the roots. If a separation of the roots is detected, then we can continue running the $QR$-iteration using the refined shifts. 
	
We denote by $P_M(m)$ the matrix $[e_1 \ Me_1 \ \ldots \ M^{m-1}e_1]$. If $B + O(p^N) \in \M_n(\Zp)$ is topologically nilpotent, the matrix $P_B(m)$ is often not given at a flat absolute precision; the $i$-th column is actually known at absolute precision $N - \log_p \norm{B^ie_1}$. By Wilkinson's lemma, columns of the matrix $P_B(m)$ can be cached during a $QR$-iteration, so the cost of constructing the matrix is negligible.

\begin{lemma}
	Let $M \in \M_n(\Zp)$ be a Hessenberg matrix. Then for all $m \geq 1$, we have that $P_M(m)$ is upper triangular, and for each $1 \leq i \leq n$, we have $\abs{P_M(m) \ind{i,i}} \geq \abs{P_M(m)\ind{i',j'}}$ for all $i',j' \geq i$.
\end{lemma}

\begin{proof}
	Triangularity is obvious. Let $\alpha := P_M(m)\ind{i,i}$ be a diagonal entry. If $\abs{\alpha} = 1$ there is nothing to do, and if $\abs{\alpha} < 1$ we have that $M^i e_1$ is a $\Zp$-linear span of $e_1, Me_1, \ldots, M^{i-1}e_1$  modulo $\alpha$. When $j' > n$, we have $M^{j'} e_1$ is a span of the columns of $P_M(n)$ by the Cayley-Hamilton Theorem. 
\end{proof}

\begin{corollary} \label{cor: DQD-factorization}
	The matrix $P_M(m)$ admits a factorization $P_M(m) = D Q_M(m)$, where $D$ is a diagonal matrix such that
	    $
	        \abs{D\ind{i,i}} \geq \abs{D \ind{i+1, i+1}}
	    $
	and where $Q_M(m) \in \GL_n(\Zp)$.
\end{corollary}

Using Corollary~\ref{cor: DQD-factorization}, we can determine an approximation to a factor of $\chi_{\smeig}$ provided that either some $D\ind{i,i}$ is very small (in which case, the orbit of $e_1$ is nearly a proper invariant subspace), or provided that no $D\ind{i,i}$ is too small (meaning the matrix $P_A(m)$ is reasonably well-conditioned). We believe that a more precise statement of what we can determine from this approximation to $\chi_{\smeig}$ is an interesting problem for future study.
%


%

%
%


\subsection{The $QR$-algorithm} \label{sec: new QR}

    We give the fast version of the $QR$-algorithm, given as Algorithm~\ref{algo: fast QR}. The conditional statement on Line~\ref{step: while still converging fast} should be interpreted as ``while the iteration is still converging super-linearly''. 
    
	
    \begin{algorithm}
		\caption{\texttt{QR\_Iteration} \qquad (Fast version)}
		\label{algo: fast QR}
		\begin{algorithmic}[1]
			\REQUIRE
			\ \\
			$H + O(p^N)$, an $n \times n$-matrix over $\mbz_p$ in {\sizesorted } Hessenberg form. \\
			$\chi_{H,p}$, the characteristic polynomial of $M \pmod p$. \\[1ex]

			\ENSURE
			A block triangular form $T$ for $H$, and matrix $V$ so that $HV = VT + O(p^N)$. \\[1ex]
			
			\STATE
			Set $m$ to be the multiplicity of $0$ in $\chi_{H,p}$.
						
			\STATE Set $[A; \epsilon, B] := H$
			
			\STATE Set $\epsilon_{\textrm{old}} := 1$
					
			\WHILE{\texttt{true}}
			
				\FOR{$j=1, \ldots, m$}
					\STATE 
					Set $\mu := m^{-1} \trace(B)$. Adjust precision if needed.
					\STATE
					Factor $QR := H - \mu I$
					\STATE
					Set $H := RQ + \mu I$
					\STATE 
					Set $[A; \epsilon, B] := H$
					\STATE
					Set $V := Q^{-1}V$ \\[1ex]
				\ENDFOR
			
			\IF {$\abs{\epsilon} > \abs{\epsilon_{\textrm{old}}}^2$} \label{step: while still converging fast}
				\RETURN \texttt{Fail}, $H$, $V$
			\ELSIF {$\epsilon = 0 \pmod{p^N}$}
				\RETURN \texttt{Success}, $H$, $V$
			\ELSE
				\STATE
				Set $\epsilon_{\textrm{old}} := \epsilon$.		
			\ENDIF
			
			\ENDWHILE
			
			
		\end{algorithmic}
	\end{algorithm}

\begin{table} 
	\centering
	\begingroup
	\renewcommand{\arraystretch}{1.3}
	\setlength{\tabcolsep}{7pt}
	\begin{tabular}{l|l|l}
		Line(s) & Cost per line (leading term) \\ \hline \hline
		4 & $ \ceil{\log_2 N}$ iterations \\
		-- 5 & $m$ iterations \\
		-- -- 7,8, \& 10 & $\frac{1}{2}n^2 + \frac{1}{2}n^2 + n^2$ & In parallel \\
		Total (main term): & $2n^2m \ceil{\log_2 N}$
	\end{tabular}
	\endgroup
	\bigskip
	\caption{Table of costs for the $QR$-algorithm.}
	\label{fig: QR cost table}
\end{table}

\begin{proposition} \label{prop: QR complexity}
	\propIterationSubroutine
\end{proposition}

	
\begin{proof}
	 The result is obtained by combining Proposition~\ref{prop: precision doubling proposition} and tabulating the costs in Table~\ref{fig: QR cost table}. 
\end{proof}


\section{The main algorithm} \label{sec: main algo}

In this section, we describe the main algorithm (Algorithm~\ref{algo: main algorithm}) and prove the main theorem. Though our main theorem is concerned with matrices whose eigenvalues are all defined in $\Qp$, we introduce some terminology to state more precisely how our algorithm performs in general

\begin{definition} \label{def: weak block Schur form}
	We say that a matrix is in \emph{weak block Schur form} if it is block upper triangular and for each block $B$, either the characteristic polynomial of $B$ has no roots in $\Qp$ or there is a $\lambda \in \Qp$ such that $B - \lambda I$ is topologically nilpotent. 
\end{definition}

Note that the weak block Schur form can be converted to a block Schur form by applying the eigenvector methods \cite{CRV2017characteristic, Kulkarni2019} to the diagonal blocks, and then applying the resulting change of basis to the whole matrix. If the characteristic polynomial of $M$ modulo $p$ is square-free and splits completely, the weak block Schur form is a Schur form.

\begin{algorithm}[h]
	\caption{Main algorithm}
	\label{algo: main algorithm}
	\begin{algorithmic}[1]
		\REQUIRE
		\ \\
		$M + O(p^N)$, an $n \times n$-matrix over $\mathbb{Q}_p$. \\[1ex]
		
		\ENSURE
		A weak block Schur form $T$ for $M$, and matrix $U$ such that $MU = UT + O(p^N)$. \\[1ex]
		
		\STATE
		Set $d := \norm{M}$, $M := d^{-1} \cdot M$
		
		\STATE \label{step: main: sorted form}
		Set $M, U := \texttt{sorted\_form(M)}$
		
		\STATE
		Set $S$ to be the block sizes of $M$

		\STATE \label{step: main: attempt sorted hessenberg}
		Set $\texttt{retcode}, M, U_1 := \texttt{attempt\_sorted\_hessenberg}(M, S)$

		\STATE
		Update $U := UU_1$
		
		\STATE
		Compute $\chi_{M,p}$
		
		\WHILE {$M$ has an eigenvalue defined over $\Zp$}
			\STATE
			Choose $\mu \in \Zp$ such that $\mu \pmod p$ is a root of the characteristic polynomial of the bottom-right block of $M \pmod p$
			 
			\STATE
			Apply one $QR$-round to $M$ with shift $\mu$
			
			\IF {$M - \mu I \pmod{p}$ is a \thingname} \label{step: main: attempt shift}
				\STATE \label{step: main: QR iteration}
				Set $\texttt{retcode2}, M, U_2 := \texttt{QR\_Iteration}(M - \mu I, \chi_{M,p})$
				\STATE
				Set $M := M + \mu I$
				\STATE
				Update $U := UU_2$
			\ELSE
				\STATE
				Set $\texttt{retcode} := \texttt{Fail}$
			\ENDIF

			\IF {$\texttt{retcode} == \texttt{Fail}$ \OR $\texttt{retcode2} == \texttt{Fail}$}
				\STATE
					Apply a fallback method (we use Algorithm~\ref{algo: simple QR}, and obtain the output $M, U_3$)
				\STATE
					Update $U := UU_3$
				\RETURN $M, U$
			\ENDIF
			
			\STATE
			Deflate $M$ to be the top-left block (thereby reducing the size of $M$)
		\ENDWHILE
				
		\STATE
		Reset $M$ to be the full-sized matrix
		
		\RETURN $d \cdot M, U$.
	\end{algorithmic}
\end{algorithm}
Note that each of the matrix multiplication steps in Algorithm~\ref{algo: main algorithm} can be combined into the preceding step, so do not actually contribute to the complexity; we separated out the update steps for clarity. 

\subsection{Proof of the Main Theorem}

We now prove our main theorem on
the behaviour of Algorithm~\ref{algo: main algorithm}
in the special case of a matrix with $n$ eigenvalues
in $\Zp$ that are simple modulo $p$.

\begin{theorem} \label{thm:main_theo}
    \mainTheorem
\end{theorem}

\begin{proof}
	After step~\ref{step: main: sorted form}, we may assume that our matrix is of the form
	\[
	M \equiv 
	\bbm
	A & \ast & \cdots & \ast \\
	& B_1 & & \vdots \\
	& & \ddots & \ast \\
	& & & B_r
	\ebm
	\pmod p,
	\]
	where $\chi_A \pmod p$ has no linear factors, and every $\chi_{B_j}(t) \equiv (t-\lambda_j)^{m_j} \pmod p$ for some $\lambda \in \mathbb{F}_p$. By our assumption on $\chi_M$, we see that the $A$ block is empty and $B_r$ is a block of size $1$ in $M \pmod p$. We see that step~\ref{step: main: attempt sorted hessenberg} will produce a \unshiftedthingname \ of the form $[A; \epsilon, b_r]$ and that the condition in step~\ref{step: main: attempt shift} is satisfied.
	By Proposition~\ref{prop: QR complexity}, step~\ref{step: main: QR iteration} will transform $M$ to a matrix of the form $[A'; 0, \lambda_r]$. Additionally, step~\ref{step: main: QR iteration} will preserve the Hessenberg form.
	
	\bigskip
	We now look at the deflated instance where $M'' = A'$. Specifically, we will show that the condition in step~\ref{step: main: attempt shift} is satisfied. Write
	\[
		M'' \equiv
		\bbm
		B_1'' & \cdots & \ast \\
		& \ddots & \vdots \\
		& & B_{r''}''
		\ebm
		\pmod p,
	\]
	where by definition the subdiagonal entries of each $B_j''$ are non-zero modulo $p$.  By the assumption on $\chi_M$, we have that $\chi_{B_{r''}''} \pmod p$ has a simple root $\bar \mu$ over $\mathbb{F}_p$. We choose a lift $\mu \in \Zp$ for $\bar \mu$. 
	
	For a Hessenberg matrix $H$, we have with $H = QR$ a $QR$-decomposition that $\abs{R\ind{i,i}} \geq \abs{H\ind{i+1, i}}$. Thus, after one $QR$-round with shift $\mu$ we have that the bottom row of $M''$ is congruent to $0$ modulo $p$. Since $\bar \mu$ is a simple root of the characteristic polynomial, we additionally have that $M$ is in sorted Hessenberg form. Thus, step~\ref{step: main: attempt shift} succeeds to produce a \thingname. We now see that the algorithm produces a Schur form for $M$ by induction.
	
	By Proposition~\ref{prop: QR complexity}, we see that each execution of step~\ref{step: main: QR iteration} consists of $\log_2 (N)$ $QR$-rounds, after which the subdiagonal $\epsilon$ converges to $0 + O(p^N)$.
	The total cost for this is $2n^2 \log (N) + o(n^2 \log(N))$.
	Since deflation reduces the number of rows/columns of the input matrix by $1$, we see repeated applications of step~\ref{step: main: QR iteration} require a total of $\frac{2}{3}n^3 \log (N) + o(n^3 \log(N))$ arithmetic operations in $\Zp$.
	Finally, to compute the eigenvectors,
	only $n$ triangular systems are to be
	solved, for a total of  $O(n^3)$
	arithmetic operations in $\Qp$ (there may be some divisions by powers of $p$).	
\end{proof}

\section{Practicality and Implementation} \label{sec: implementations}

	In this section, we give some timings for our \texttt{Julia} implementation, available at:
		\begin{center}
			\url{https://github.com/a-kulkarn/Dory}
		\end{center}
	Our benchmarking results are listed in Tables~\ref{table: timings 7-10} and~\ref{table: timings 41-100}. We also include the old timings from \cite{Kulkarni2019} for the sake of reference (Table~\ref{table: timings old}), however, the updates to the dependencies and the change in hardware means the comparison is not pure. Timings are based on random matrices, where each entry is a randomly sampled $p$-adic number in \texttt{PadicField($p$,$N$)} (more precisely, a uniformly random integer in $[0,p^{N}-1]$).
	
	\begin{table}[h]
		\centering
		\begin{tabular}{crrr}
			Matrix size ($n$) & Time (s) (power iteration) & Time(s) (block schur form) & Time (s) (classical) \\ \hline \\[-1em]
			10   & 0.0029  & 0.010 & 0.0008   \\
			100  & 0.9774  & 3.390 & 3.2600   \\
			200  & 6.7920  & 24.2771 & 51.2573  \\
			300  & 36.0114 & 166.4447 & 258.0104 \\
		\end{tabular} 
		\bigskip
		\caption{Timings from \cite{Kulkarni2019}. ($\Qp := \texttt{PadicField(7,10)}$)} \label{table: timings old}
	\end{table}

	\begin{table}[h]
		\begin{tabular}{crrr}
			Matrix size ($n$) & Time (s) (power iteration) & Time(s) (block schur form) & Time (s) (classical) \\ \hline \\[-1em]
			10   & 0.0017  & 0.0524 & 0.0006   \\
			100  & 0.5386  & 1.4558 & 2.0400   \\
			200  & 3.7068  & 10.1043 & 31.1456  \\
			300  & 20.4178 & 52.0343 & 158.4332 \\
		\end{tabular}
		\bigskip
		\caption{Timings with improved $QR$. ($\Qp := \texttt{PadicField(7,10)}$, simple roots over $\Fp$)} \label{table: timings 7-10}
	\end{table}
	
	\begin{table}[h]
		\centering
		\begin{tabular}{crrr}
			Matrix size ($n$) & Time (s) (power iteration) & Time(s) (block schur form) & Time (s) (classical) \\ \hline \\[-1em]
			10   & 0.0125  & 0.0196 & 0.0060   \\
			100  & 6.9100  & 14.9795 & 19.7082   \\
			200  & 44.5217 & 39.6243 & 337.6393  \\
		\end{tabular}
		\bigskip
		\caption{Timings with improved $QR$, more precision. ($\Qp := \texttt{PadicField(41,100)}$)} \label{table: timings 41-100}
	\end{table}

	Timings were conducted by using the \texttt{time()} function. An average of $10$ samples were used per comparison, with each method receiving the same inputs. We omit from the timings an extra execution of each function at the beginning which triggers \texttt{Julia}'s compiler. The code to execute the comparisons is found in \texttt{Dory/test/timings.jl} and \texttt{Dory/test/timings2.jl}.

\section*{Acknowledgements}
The authors would like to thank the mathematics department at TU Kaiserslautern for sponsoring the visit of the second author. We would also like to thank Eran Assaf and John Voight for their especially insightful comments.


\end{document}